\documentclass[11pt]{article}
\usepackage{amsmath, amsthm, amscd, amssymb}
\usepackage{tikz,cite}
\usetikzlibrary{positioning,automata,backgrounds}
\usepackage{caption,subcaption}
\usepackage{hyperref}
\usepackage{relsize}
\numberwithin{equation}{section}

\theoremstyle{plain}
\newtheorem{theorem}{Theorem}[section]
\newtheorem*{cor-acyclic}{Corollary~\ref{cor:acyclic}}

\newtheorem{proposition}[theorem]{Proposition}
\newtheorem{lemma}[theorem]{Lemma}
\newtheorem{corollary}[theorem]{Corollary}
\newtheorem{claim}[theorem]{Claim}

\newtheorem{fact}[theorem]{Fact}
\newtheorem{prob}[theorem]{Problem}

\theoremstyle{definition}

\theoremstyle{plain}

\def\cH{\mathcal{H}}

\begin{document}
\title{Characterization of forbidden subgraphs for bounded star chromatic number}

\author{
Ilkyoo Choi\thanks{
Department of Mathematics, Hankuk University of Foreign Studies, Yongin-si, Gyeonggi-do, Republic of Korea.
E-mail: \texttt{ilkyoo@hufs.ac.kr}
}
\and
Ringi Kim\thanks{
Corresponding author. Department of Mathematical Sciences, KAIST, Daejeon, Republic of Korea.
E-mail: \texttt{ringikim2@gmail.com}
}
\and
Boram Park\thanks{
Department of Mathematics, Ajou University, Suwon-si, Gyeonggi-do, Republic of Korea.
E-mail: \texttt{borampark@ajou.ac.kr}
}
}

\date\today

\maketitle

\begin{abstract}
The {\it chromatic number} of a graph is the minimum $k$ such that the graph has a proper $k$-coloring.
There are many coloring parameters in the literature that are proper colorings that also forbid bicolored subgraphs.
Some examples are $2$-distance coloring, acyclic coloring, and star coloring, which forbid a bicolored path on three vertices, bicolored cycles, and a bicolored path on four vertices, respectively.
This notion was first suggested by Gr\"unbaum in 1973, but no specific name was given.
We revive this notion by defining an {\it $H$-avoiding $k$-coloring} to be a proper $k$-coloring that forbids a bicolored subgraph $H$.

When considering the class $\mathcal C$ of graphs with no $F$ as an induced subgraph, it is not hard to see that every graph in $\mathcal C$ has bounded chromatic number if and only if $F$ is a complete graph of size at most two.
We study this phenomena for the class of graphs with no $F$ as a subgraph for $H$-avoiding coloring.
We completely characterize all graphs $F$ where the class of graphs with no $F$ as a subgraph has bounded $H$-avoiding chromatic number for a large class of graphs $H$.
As a corollary, our main result implies a characterization of graphs $F$ where the class of graphs with no $F$ as a subgraph has bounded star chromatic number.
We also obtain a complete characterization for the acyclic chromatic number.
\end{abstract}

\section{Introduction}\label{sec-intro}
All graphs in this paper are simple and have no loops.
Given a graph $G$, a {\it proper $k$-coloring} of $G$ is a partition of its vertex set $V(G)$ into $k$ parts such that there is no edge with both endpoints in the same part.
In other words, there exists a function $\varphi$ on $V(G)$ to a set of $k$ colors such that $\varphi(x)\neq \varphi(y)$ for each edge $xy$ of $G$.
The {\it chromatic number} of $G$, denoted $\chi(G)$, is the minimum $k$ such that $G$ has a proper $k$-coloring.

Since a proper coloring of a graph $G$ is a partition of its vertex set, there is an intimate relation between the structure of $G$ and its chromatic number $\chi(G)$.
Since a complete graph requires all of its vertices to receive different colors, the size of a largest clique of a graph $G$, the clique number of $G$, is a trivial lower bound on $\chi(G)$.
There was great interest in characterizing graphs where the aforementioned trivial lower bound is also an upper bound for the chromatic number;
these graphs are also known as {\it perfect graphs}.
Conjectured by Berge~\cite{1961Be} in 1961, the characterization of perfect graphs in terms of forbidden induced subgraphs was completed by Chudnovsky et al.~\cite{2006ChRoSeTh}.
It is now known as the Strong Perfect Graph Theorem, which states that a graph is a perfect graph if and only if it contains neither an odd cycle nor the complement of an odd cycle as an induced subgraph.

As a generalization of perfect graphs, researchers started investigating which classes of graphs with a forbidden structure are {\it $\chi$-bounded}; a class $\mathcal C$ of graphs is $\chi$-bounded if there exists a function such that the chromatic number of each graph $G$ in $\mathcal C$ is bounded by the same function of the clique number of $G$.
By the random construction of Erd\H{o}s~\cite{1959Er}, for each $k$ and $g$, we know that there exists a graph $G$ with $\chi(G)\geq k$ and minimum cycle length at least $g$.
This implies that for the class of graphs without $F$ as an induced subgraph to be $\chi$-bounded, it is necessary that $F$ is a forest.
Gy\'arf\'as~\cite{1975Gy} and Sumner~\cite{1981Su} independently conjectured that the necessary condition is also sufficient, namely, for a forest $F$ and an integer $q$, there is a function $f(F, q)$ such that the chromatic number of a graph that contains neither $F$ as an induced subgraph nor a clique of size $q$ is bounded by $f(F, q)$.
This conjecture is known to be true for paths~\cite{1985Gy}, stars, and some additional forests~\cite{1980Gy,1994Ki,2004Ki,1997Sc,2017Ch}, but is mostly wide open.
We redirect the readers to~\cite{2016ChScSe,2016ScSe} for more information regarding recent progress on various conjectures regarding $\chi$-bounded graphs.

A natural question to ask is changing the above question from forbidding two structures (a graph $F$ and a clique of size $q$) to a single structure.
When the structure is defined to be an induced subgraph, since $F$ must be both a complete graph and a forest, we obtain the following fact:

\begin{fact}\label{fact-induced}
The class of graphs without $F$ as an induced subgraph is $\chi$-bounded if and only if $F$ is a complete graph on at most two vertices.
\end{fact}

We are interested in when the forbidden structure is a subgraph, which is not necessarily induced.
We define an {\it $F$-free graph} to be a graph that does not contain $F$ as a subgraph.
As in the case of induced subgraphs, in order for the class of $F$-free graphs to be $\chi$-bounded, it must be the case that $F$ is a forest by the previously mentioned construction by Erd\H{o}s.
However, since a graph with minimum degree at least $k-1$ contains all trees of order $k$, we know that for a tree $F$, an $F$-free graph must contain a vertex of degree at most $k-2$.
This further implies that $G$ is $(k-2)$-degenerate, and an ordering of the vertices exists such that each vertex has at most $k-2$ later neighbors.
Since we can greedily color the vertices in the reverse order with $k-1$ colors, we know that $G$ is $(k-1)$-colorable.
Therefore, we conclude the following:

\begin{fact}\label{fact-subgraph}
The class of $F$-free graphs is $\chi$-bounded if and only if $F$ is a forest.
\end{fact}

\subsection{Our contributions}\label{subsec:intro}

In this paper, we are interested in obtaining characterizations similar to Fact~\ref{fact-induced} for various other coloring parameters.
Instead of dealing with each coloring parameter separately, we define a generalized coloring parameter that captures the constraints of several well-known coloring parameters in the literature such as $2$-distance coloring, acyclic coloring, and star coloring.
All these coloring parameters are proper colorings that also forbid bicolored subgraphs;
in addition to being a proper coloring, $2$-distance coloring, acyclic coloring, and star coloring forbids a bicolored path on three vertices, bicolored cycles, and a bicolored path on four vertices, respectively.
In particular, acyclic coloring and star coloring have been an active area of research.
Yet, as far as we know, there has been no attempt in trying to characterize the forbidden structures to obtain bounded acyclic chromatic number or star chromatic number.
We direct the readers to the well-known graph coloring book~\cite{1995JeTo} of Jensen and Toft for more information regarding various coloring parameters.

This notion of proper colorings that also forbid bicolored subgraphs was suggested by Gr\"unbaum~\cite{1973Gr} in 1973, but no specific name was given.
We revive this notion by introducing the concept of {\it $H$-avoiding $k$-coloring}.
To be precise, given a graph $G$ and a (smaller) connected graph $H$, an {\it $H$-avoiding $k$-coloring} of $G$ is a proper $k$-coloring with the additional constraint that there is no bicolored $H$.
For example, a $P_3$-avoiding coloring is exactly the case of $2$-distance coloring, which is equivalent to proper coloring the square of the graph, and a $P_4$-avoiding coloring is exactly the case of star coloring. (Here, we denote by $P_n$ the path graph on $n$ vertices.)
Note that if $H$ is not $2$-colorable, then $H$-avoiding coloring is the same as proper coloring.

We say a class $\mathcal C$ of graphs has bounded $H$-avoiding chromatic number if there is a function $f(H)$ such that each graph in $\mathcal C$ has an $H$-avoiding $f(H)$-coloring.
The purpose of this paper is to investigate the phenomena in Fact~\ref{fact-induced} and~\ref{fact-subgraph} for $H$-avoiding coloring.
Namely, we initiate the study of characterizing all graphs $F$ for a given $H$ where the class of $F$-free graphs has bounded $H$-avoiding chromatic number.
We start with the following simple result regarding stars:

\begin{theorem}
Let $H$ be a star.
The class of $F$-free graphs has bounded $H$-avoiding chromatic number if and only if $F$ is a star.
\end{theorem}
\begin{proof}
Let $H$ be a star with maximum degree $d$ and let $G$ be an $F$-free graph.
Suppose that the class of $F$-free graphs has bounded $H$-avoiding chromatic number with $C$ colors.
This implies that a vertex in $G$ has at most $(C-1)(d-1)$ neighbors, so $F$ must be a star with maximum degree $(C-1)(d-1)+1$.
Now suppose $F$ is a star with maximum degree $D$, which implies that $G$ has maximum degree $D-1$.
By greedily coloring the vertices of $G$ so that each vertex avoids all colors of vertices within distance $2$, we obtain an $H$-avoiding coloring of $G$.
\end{proof}

Throughout this paper, we say a vertex is a {\it $d$-vertex} and a {\it $d^+$-vertex} if the degree of $v$ is exactly $d$ and at least $d$, respectively.
Our main result is the following theorem that deals with a much larger class of bipartite graphs:

\begin{theorem}\label{thm-main}
Let $H$ be a connected bipartite graph that is not a star where all vertices in one part have degree at most $2$.
The class of $F$-free graphs has bounded $H$-avoiding chromatic number if and only if $F$ is a forest where each pair of $3^+$-vertices has even distance.
\end{theorem}

The following corollary is a direct consequence of our main result:

\begin{corollary}\label{cor:acyclic}
For a graph $F$,  the class of $F$-free graphs has bounded star chromatic number if and only if
$F$ is a forest where each pair of $3^+$-vertices has even distance.
\end{corollary}

The following result is not a direct consequence of Theorem~\ref{thm-main}, nonetheless it follows from our lemmas.
The proof of the following corollary is in Section~\ref{sec:bounded}.

\begin{corollary}\label{cor:acyclic}
For a graph $F$,  the class of $F$-free graphs has bounded acyclic chromatic number if and only if
$F$ is a forest where each pair of $3^+$-vertices has even distance.
\end{corollary}

Let $H$ be a connected bipartite graph that is not a star where all vertices in one part have degree at most $2$.
In Section~\ref{sec:unbounded} we construct a family of $F$-free graphs that does not have bounded $H$-avoiding chromatic number where $F$ is a forest that has some pair of $3^+$-vertices with odd distance.
On the other hand, in Section~\ref{sec:bounded} we prove that the class of $F$-free graphs has bounded $H$-avoiding chromatic number if $F$ is a forest where each pair of $3^+$-vertices has even distance.
We conclude the paper by making some remarks and suggesting future research directions in Section~\ref{sec:remarks}.


\section{Graphs with large $H$-avoiding chromatic number}\label{sec:unbounded}

Let $H$ be a connected bipartite graph that is not a star where all vertices in one part have degree at most 2.
In this section, we construct a family of $F$-free graphs that does not have bounded $H$-avoiding chromatic number where $F$ is a forest that has some pair of $3^+$-vertices with odd distance.

For an integer $n$, let $[n]=\{1,\ldots,n\}$.
For integers $m$ and $n$, let $G_{m,n}$ be the graph obtained from a complete graph on $m$ vertices by replacing each edge with $n$ disjoint paths of length 2, that is,
\begin{eqnarray*}
V(G_{m,n})&=&\{x_i: i\in [m] \}\cup\{y_{ij}^k: i,j\in[m], k\in[n], \text{ and } i<j\}\\
E(G_{m,n})&=&\{x_ly^k_{ij}: i,j,l\in [m],  k\in[n], l\in\{i,j\}, \text{ and } i<j \}.
\end{eqnarray*}
See Figure~\ref{fig:example} for an illustration of $G_{4, 5}$.
Note that each vertex $y_{ij}^k$ is a 2-vertex and each vertex $x_i$ is a $3^+$-vertex.
Moreover, two vertices $x_i$ and $x_j$ have even distance, which implies that the distance of every pair of $3^+$-vertices is even.
Hence, $G_{m,n}$ is $F$-free, if $F$ is a forest that has some pair of $3^+$-vertices with odd distance.

\begin{figure}[h]
	\begin{center}
  \includegraphics[scale=1]{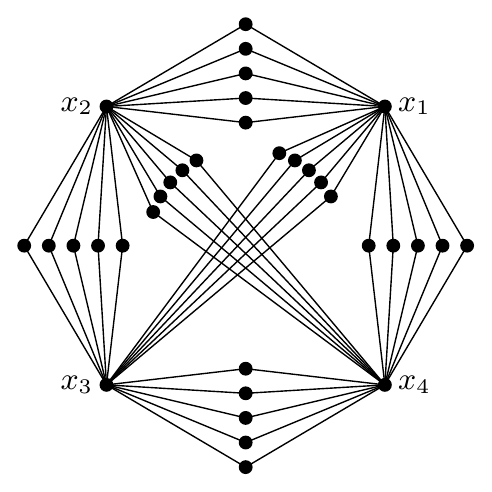}
\caption{The graph $G_{4,5}$.}\label{fig:example}	\end{center}
\end{figure}

\begin{lemma}\label{lem:Gmn-unbouned}
Let $N$ be a positive integer and let $H$ be a connected bipartite graph that is not a star where all vertices in one part have degree at most 2.
For sufficiently large $m$ and $n$, every proper $N$-coloring of $G_{m,n}$ contains a bicolored $H$.
\end{lemma}

\begin{proof}
Assume $H$ is a connected bipartite graph that is not a star with parts $P$ and $Q$ where all vertices in $Q$ have degree at most 2.
Let $|P|=p$ and $|Q|=q$.
Let $m=m(N, p)$ be sufficiently large and let  $n\geq Nq$.
For $G_{m,n}$, let $(X,Y)$ be a bipartition where $X$ and $Y$ is the set of $3^+$-vertices and $2$-vertices, respectively, of $G_{m, n}$.

Let $\varphi$ be a proper $N$-coloring of $G_{m,n}$ and define $K$ to be an auxiliary complete graph on the vertex set $X$.
We define $\phi$ to be an $N$-edge-coloring of $K$ such that each edge $x_ix_j\in E(K)$ receives a color that appears the most in $\varphi(y_{ij}^1), \ldots, \varphi(y_{ij}^n)$.
By Ramsey's Theorem~\cite{1930Ra}, since $|X|$ is sufficiently large, there exists a $(p N)$-subset $X'$ of $X$ such that every edge in $K[X']$ is monochromatic under $\phi$.
Thus, by the Pigeonhole Principle and the definition of $\phi$, each pair of vertices in $X'$ has at least $\frac{n}{N}\ge q$ common neighbors who received the same color under $\varphi$.
Without loss of generality, we may assume $\varphi(y_{ij}^1)=\cdots=\varphi(y_{ij}^{q})=1$ for each $x_i,x_j\in X'$ where $i<j$.

Moreover, since at least $\frac{|X'|}{N}=p$ vertices of $X'$ received the same color under $\varphi$, we may assume that $\varphi(x_i)=2$ for every $i \in [p]$.
Then the subgraph of $G_{m,n}$ induced by the set $\{x_i: i\in [p]\}\cup\{y_{ij}^k: i,j\in[p], i<j, k\in[q]\}$ is bicolored (using colors $1$ and $2$) under $\varphi$ and it contains a bicolored $H$.
This completes the proof.
\end{proof}

Lemma~\ref{lem:Gmn-unbouned} implies the following.

\begin{proposition}\label{lem:free-unbounded}
Let $H$ be a connected bipartite graph that is not a star where all vertices in one part have degree at most 2.
If $F$ is a graph such that the class of $F$-free graphs has bounded $H$-avoiding chromatic number,
then $F$ is a forest where each pair of $3^+$-vertices has even distance.
\end{proposition}

\begin{proof}
As noted in Section~\ref{subsec:intro}, it must be that $F$ is a forest.
By Lemma~\ref{lem:Gmn-unbouned}, for every integer $N$ and every forest $F$ where there are two $3^+$-vertices whose distance is odd, if $m$ and $n$ are sufficiently large, then $G_{m,n}$ is an $F$-free graph whose $H$-avoiding chromatic number is larger than $N$.
Hence, $F$ must be a forest where each pair of $3^+$-vertices has even distance.
\end{proof}


\section{Graphs with bounded $H$-avoiding chromatic number}\label{sec:bounded}

In this section, we prove that for every bipartite graph $H$ and every forest $F$ given in Theorem~\ref{thm-main},  the class of $F$-free graphs has bounded $H$-avoiding chromatic number.
Since $P_4$ is a subgraph of every such $H$, the $P_4$-avoiding chromatic number of a graph $G$ is at least the $H$-avoiding chromatic number of $G$.
Therefore, it is sufficient to show that the class of $F$-free graphs has bounded $P_4$-avoiding chromatic number.
As mentioned in Section~\ref{subsec:intro}, a star coloring is equivalent to a $P_4$-avoiding coloring, thus, we will prove that every $F$-free graph has bounded star chromatic number.
Recall that for a graph $G$, a proper coloring of $G$ is a \emph{star coloring} of $G$ if every bicolored subgraph of $G$ is a star forest.

\begin{proposition}\label{lem:free-bounded}
For every forest $F$ where each pair of $3^+$-vertices has even distance, there exists a constant $c=c(F)$ such that every $F$-free graph has star chromatic number at most $c$.
\end{proposition}

Before proving Proposition~\ref{lem:free-bounded}, we first give a proof of Corollary~\ref{cor:acyclic}.
By Lemma~\ref{lem:Gmn-unbouned} and Proposition~\ref{lem:free-bounded}, one can see that we also obtained a characterization for bounded acyclic coloring.
Recall that for a graph $G$, a proper coloring $G$ is an \emph{acyclic coloring} of $G$ if every bicolored subgraph of $G$ is a forest.

\begin{cor-acyclic}
For a graph $F$,  the class of $F$-free graphs has bounded acyclic chromatic number if and only if
$F$ is a forest where each pair of $3^+$-vertices has even distance.
\end{cor-acyclic}
\begin{proof}[Proof of Corollary~\ref{cor:acyclic}]
The only if part follows from Lemma~\ref{lem:Gmn-unbouned} by letting $H$ be a $4$-cycle.
The if part follows from the fact that every $k$-star coloring is also a $k$-acyclic coloring.
\end{proof}

We start with simple observations on paths and stars.

\begin{lemma}\label{lem:stars}
For a positive integer $n$,
if a graph $G$ is $K_{1,n}$-free, then $G$ is $((n-1)^2+1)$-star colorable.
\end{lemma}
\begin{proof}
Since we can color each component separately, we may assume $G$ is connected.
Let $G$ be a $K_{1,n}$-free graph, so $G$ has maximum degree $\Delta(G)$ at most $n-1$.
Thus, $\Delta(G^2)\le (n-1)^2$, where the square $G^2$ of $G$ is the graph obtained from $G$ by joining two vertices with distance at most two.
By using coloring the vertices in a greedy fashion, it follows that $G^2$ is $((n-1)^2+1)$-colorable.
Since a proper coloring of $G^2$ gives a star coloring of $G$, we conclude that $G$ is $((n-1)^2+1)$-star colorable.
\end{proof}

\begin{lemma}\label{lem:paths}
For a positive integer $n$,
if a graph $G$ is $P_n$-free then $G$ is $(n-1)$-star colorable.
\end{lemma}
\begin{proof}
Since we can color each component separately, we may assume $G$ is connected.
Let $G$ be a $P_n$-free graph.
Pick an arbitrary vertex $v$ of $G$, and let $T$ be a spanning tree obtained from a depth first search algorithm. For a vertex $u$ of $G$, let the level of $u$, denoted $\varphi(u)$, be the distance from $v$ to $u$ in $T$.
Since $G$ is $P_n$-free (so is $T$), each vertex has level at most $n-2$.
We claim that $\varphi$ is a star coloring of $G$.

First, by the definition of a depth first search tree, every pair of vertices with the same level is non-adjacent, so $\varphi$ is a proper coloring.
For vertices $u_1,u_2 \in V(G)$, if $u_1$ is on the path from $v$ to $u_2$ in $T$, then we say $u_2$ is a descendant of $u_1$ in $T$.
Suppose there is a path $P_3:u_1u_2u_3$ in $G$ with $\varphi(u_1)=\varphi(u_3)$.
Then $u_1$ and $u_3$ are in the same level, and so $u_1$ and $u_3$ are descendants of $u_2$ by the definition of a depth first search tree.
Therefore, if there is a bicolored $P_4:v_1v_2v_3v_4$ such that $\varphi(v_1)=\varphi(v_3)$ and $\varphi(v_2)=\varphi(v_4)$, then $v_2$ is a descendant of $v_3$ and $v_3$ is a descendant of $v_2$, which contradicts that $v_2$ and $v_3$ are distinct vertices.
Hence, there is no bicolored $P_4$, so $\varphi$ is an $(n-1)$-star coloring of $G$.
\end{proof}

We will typically use $\cH$ to denote a hypergraph.
We say a hypergraph is \emph{simple} if it does not have nested hyperedges, that is, there are no two distinct hyperedges $e$ and $e'$ where $e\subseteq e'$.
The \textit{degree} $\deg_{\cH}(v)$ of a vertex $v$ is the number of hyperedges of $\cH$ containing $v$.
Let $\delta(\cH)$ be the minimum degree of (a vertex of) $\cH$.

For a graph $G$ and a hypergraph $\cH$, we say
a pair $(V',E')$ where $V'\subseteq V(\cH)$ and $E'\subseteq E(\cH)$
is a \emph{$G$-skeleton} of $\cH$ if there are bijections $f: V(G)\rightarrow V'$ and
$g: E(G)\rightarrow E'$ such that $uv\in E(G)$ if and only if $\{ f(u),f(v) \} \subseteq g(uv)$.
For example, the hypergraph $\cH$ in Figure~\ref{fig:skeleton} has a $C_4$-skeleton formed by $(\{v_1,v_2,v_3,v_4\}, \{e_1,e_2,e_4,e_5\})$.

\begin{figure}
\centering
\includegraphics[scale=1]{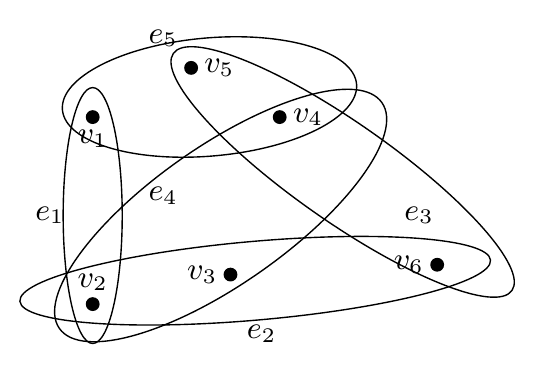}
\quad\qquad\includegraphics[scale=1]{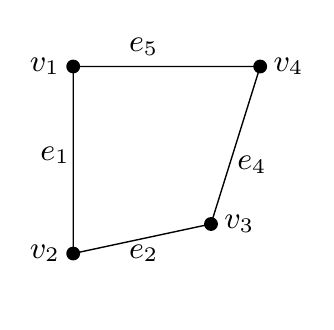}
\caption{A $C_4$-skeleton (right) of a hypergraph $\cH$ (left).}\label{fig:skeleton}
\end{figure}

\begin{lemma}\label{prop:skeleton}
Let $p(2)=1$ and $p(n)={n-2 \choose \lceil \frac{n-2}{2} \rceil }+n-1$ for $n\ge 3$.
For a nontrivial tree $T$ with $n$ vertices, if $\cH$ is a simple hypergraph with $\delta(\cH)\ge p(n)$, then $\cH$ has a $T$-skeleton.
\end{lemma}

\begin{proof}
We show by induction on $|V(T)|$.
If $|V(T)|=2$, then $T$ is a path on two vertices.
Since $\delta(\cH)\ge 2$, there exist two distinct vertices $u$, $v$ of $\cH$ and an hyperedge $e$ of $\cH$ such that $u,v \in e$.
Thus, $(\{u,v\},\{e\})$ forms a $T$-skeleton.

Suppose that the lemma is true for all trees with $n-1$ vertices for some $n\ge 3$.
Fix a tree $T$ with $n$ vertices and let $uv$ be an edge of $T$ where $v$ is a leaf.
Let $T'=T-v$ and consider a simple hypergraph $\cH$ with $\delta(\cH)\ge p(n)$.
By the induction hypothesis, $\cH$ has a $T'$-skeleton since $\delta(\cH)\geq p(n)>p(n-1)$.
Let a pair of maps $f:V(T')\to V'$ and $g:E(T')\to E'$ give a $T'$-skeleton $(V', E')$ of $\cH$.
If there exists $e\in E(\cH)\setminus E'$ such that $f(u)\in e$ and $e\setminus V' \neq \emptyset$, say $x\in e\setminus V'$, then $(V'\cup \{x\}, E'\cup  \{e\})$ forms a $T$-skeleton of $\cH$;
in other words, let $f(v)=x$ and $g(uv)=e$.
Therefore, it is sufficient to prove that there exists such a hyperedge $e \in E(\cH)$.

Since $\cH$ is simple, there exist at most ${n-2 \choose \lceil \frac{n-2}{2} \rceil }$ hyperedges contained in $V'$ that contain $f(u)$ by Sperner's Theorem~\cite{1928Sp}.
Since there are at least $p(n)$ hyperedges in $\cH$ containing $f(u)$, $|E'|=n-2$, and $p(n)-{n-2 \choose \lceil \frac{n-2}{2} \rceil}-|E'|>0$, it follows that there is a hyperedge $e \in E(\cH)\setminus E'$ such that $f(u)\in e$ and $e\setminus V'\neq \emptyset$.
This completes the proof.
\end{proof}

For a nonempty hypergraph $\cH$, the {\it rank} of $\cH$ is the size of a largest hyperedge of $\cH$.
For a hypergraph $\cH$, a \emph{rainbow $n$-coloring} of $\cH$ is an $n$-coloring of the vertices of $\cH$ such that if two vertices receive the same color, then there is no hyperedge that contains both vertices.

\begin{lemma}\label{lem:rainbow}
Let $n$ and $r$ be positive integers at least $2$ where $p(n)$ is defined as in Lemma~\ref{prop:skeleton} and $q(n,r)=(p(n)-1)(r-1)+1$.
For an $n$-vertex tree $T$ and a simple hypergraph $\cH$ with rank at most $r$,
if $\cH$ has no $T$-skeleton, then $\cH$ has a rainbow $q(n,r)$-coloring.
\end{lemma}

\begin{proof}
Let $T$ be an $n$-vertex tree.
We prove by induction on $|V(\cH)|$.
It is obvious when $|V(\cH)|=1$.
Fix a simple hypergraph $\cH$ where $|V(\cH)|>1$.
Suppose that the lemma is true for every hypergraph with less than $|V(\cH)|$ vertices.
If $\delta(\cH) \ge p(n)$ then $\cH$ has a $T$-skeleton by Lemma~\ref{prop:skeleton}, which is a contradiction.
Therefore, there must be a vertex $v$ whose degree is at most $p(n)-1$.
Now, let $\cH'$ be the following hypergraph:
\begin{eqnarray*}
V(\cH')&=&V(\cH)\setminus\{v\}\\
E(\cH')&=&\{e-v : e\in E(\cH)\}.
\end{eqnarray*}
Let $\cH''$ be a maximal simple hypergraph contained in $\cH'$, which may be obtained by removing every hyperedge contained in some hyperedge of $\cH'$ one by one.
Since $\cH''$ is a simple hypergraph with rank at most $r$ and has no $T$-skeleton satisfying $|V(\cH'')|<|V(\cH)|$,
by the induction hypothesis, $\cH''$ has a rainbow $q(n,r)$-coloring $\varphi$.
Since each hyperedge of $\cH$ has size at most $r$, the vertex $v$ has at most $(p(n)-1)(r-1)$ neighbors in $\cH$.
From the fact that $q(n,r)=(p(n)-1)(r-1)+1$, there exists an available color $\alpha$ for $v$.
Then, by extending $\varphi$ to also include $\varphi(v)=\alpha$, we obtain a rainbow $q(n,r)$-coloring of $\cH$.
This completes the proof.
\end{proof}

\bigskip

Now we are ready to give a proof of Proposition~\ref{lem:free-bounded}.

\begin{proof}[Proof of Proposition~\ref{lem:free-bounded}]
For every forest $F$ where each pair of $3^+$-vertices has even distance, there exists a tree $T$ containing $F$ as a subgraph where each pair of $3^+$-vertices of $T$ has even distance.
Therefore, it is sufficient to prove the proposition for every tree $T$ where each pair of $3^+$-vertices has even distance.

We show by induction on $|V(T)|$.
If $T$ has at most four vertices, then since $T$ is either a path or a star, the proposition holds by Lemmas~\ref{lem:stars} and \ref{lem:paths}.

Suppose that the proposition is true for all trees with at most $n-1$ vertices where $n\ge 5$.
Let $T$ be an $n$-vertex tree where each pair of $3^+$-vertices has even distance.
If $T$ has no $3^+$-vertex, then $T$ is a path and the proposition holds by Lemma~\ref{lem:paths}.
Suppose that $T$ has a $3^+$-vertex, and we define a tree $T^*$ as the following:
The vertex set $V(T^*)$ is the set of all $3^+$-vertices and the vertices whose distance to a $3^+$-vertex is even.
Note that for a vertex $v$ of $T$, the parity of the distance from $v$ to an arbitrary $3^+$-vertex is the same, since each pair of $3^+$-vertices of $T$ has even distance.
We add an edge between two vertices of $V(T^*)$ if and only if their distance is two.
See Figure~\ref{fig:T-star} for an illustration of $T$ and $T^*$.
Note that $T^*$ is also a tree since each pair of $3^+$-vertices of $T$ has even distance.

\begin{figure}
\centering
\includegraphics[scale=1.15]{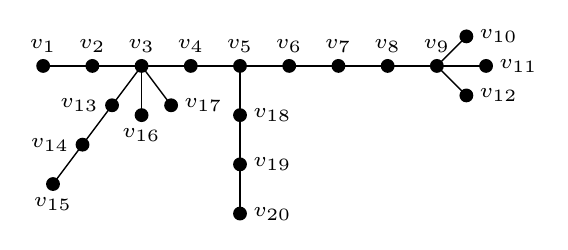}
\includegraphics[scale=1.15]{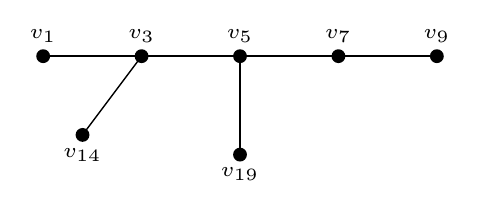}
\caption{Trees $T$ (left) and $T^*$ (right).}\label{fig:T-star}
\end{figure}

Consider a $T$-free graph $G$.
Let $d=|V(T^*)|\cdot \Delta(T)+n$, and let $V_B=\{v\in V(G): \deg_\cH(v)>d\}$ and $V_S=V(G)\setminus V_B$.
Let $vv'$ be an edge of $T$ where $v$ is a pendent vertex of $T$ and let $T'=T-v$.
Note that a pair of $3^+$-vertices in $T'$ still has even distance.
By the induction hypothesis, there is a positive integer $c(T')$ such that a $T'$-free graph is $c(T')$-star colorable.
We let
\[c(T)= c(T')\cdot q(|V(T^*)|,d)+  d^2+1,\]
where $q(|V(T^*)|,d)$ is defined as in Lemma~\ref{lem:rainbow}.
We will show that $G$ is $c(T)$-star colorable.
It is clear that $G[V_S]$ is $K_{1,d+1}$-free, so $G[V_S]$ has a $(d^2+1)$-star coloring $\varphi_S$ by Lemma~\ref{lem:stars}.

For the graph $G[V_B]$, we will define two star colorings $\varphi_B^1$ and $\varphi^2_B$ that will be utilized to obtain a star coloring of $G$.
We first prove two structural lemmas regarding the graph $G[V_B]$.

\begin{claim}\label{claim:noT'}
The graph $G[V_B]$ is $T'$-free.
\end{claim}
\begin{proof}
Suppose to the contrary that $G[V_B]$ contains $T'$ as a subgraph.
Let $G'$ be a subgraph of $G[V_B]$ isomorphic to $T'$, and let $u'$ be the vertex of $G'$ corresponding to the vertex $v'$ of $T'$. Since the degree of $u'$ is larger than $d$, which is larger than $|V(T')|$, we know that $u'$ has a neighbor $u''$ in $G\setminus V(G')$.
Now, by adding the vertex $u''$ and the edge $u'u''$ to $G'$, we obtain a subgraph of $G$ isomorphic to $T$, which is a contradiction.
\end{proof}

Let $\cH'$ be a hypergraph where $V(\cH')=V_B$ and
$  E(\cH')=\{ N_G(x)\cap V_B : x\in V_S \}$, and let $\cH$ be a simple hypergraph obtained from $\cH'$ by
sequentially deleting a hyperedge $e'$ of $\cH' $ where there exists an edge $e$ such that $e'\subseteq e$.
Note that $\cH$ must be a simple hypergraph.
Since each vertex of $V_S$ has degree at most $d$ in $G$, the rank of $\cH$ is at most $d$.
See Figure~\ref{fig:hypergraph:proof} for an illustration.

\begin{figure}
\begin{center}
\includegraphics[scale=1]{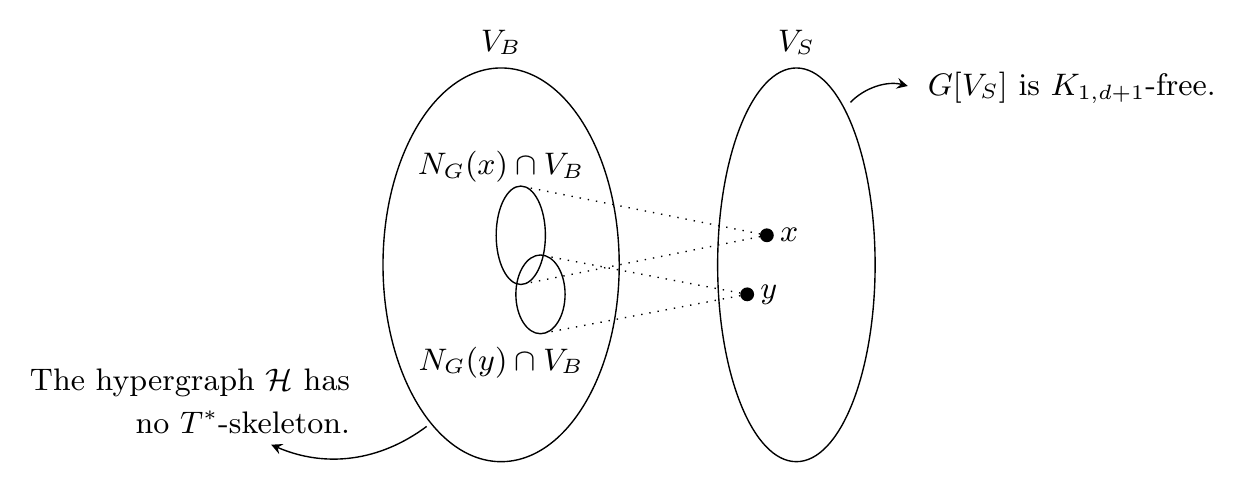}
\caption{An illustration for the proof}\label{fig:hypergraph:proof}
\end{center}
\end{figure}

\begin{claim}\label{claim:noT*}
The hypergraph $\cH$ has no $T^*$-skeleton.
\end{claim}
\begin{proof}
Suppose to the contrary that $(V^*,E^*)$ is a $T^*$-skeleton of $\cH$.
We will find a subset $A$ of $V(G)$ where $G[A]$ contains a subgraph that is isomorphic to $T$, which is a contradiction.
Since $T$ is an $n$-vertex tree, we may assume that $V^*=\{t_1, \ldots, t_n\}$.

For each hyperedge $e\in E^*$, there is a corresponding vertex $x_e\in V_S$ such that $e=N_G(x_e)\cap V_B$.
Since $|V^*|=|V(T^*)| \le \frac{|V(T)|}{2}$ and $|E^*|=|E(T^*)|=|V(T^*)|-1\le \frac{|V(T)|}{2}-1$, we know that $|V^*\cup \{x_e: e\in E^*\}|\leq n-1$.
Therefore, since each vertex $t\in V^*$ has more than $|V(T^*)|\cdot \Delta(T)+n$ neighbors, it follows that
$|N_G(t)\setminus (V^*\cup\{x_e: e\in E^*\})|\ge |V(T^*)|\cdot \Delta(T)$.

Let $N_1$ be a $\Delta(T)$-set from $N_G(t_i)\setminus (V^*\cup\{x_e: e\in E^*\})$ and for each $i\in[n]\setminus\{1\}$, let $N_i$ be a set of $\Delta(T)$ vertices from
\[N_G(t_i)\setminus \left(V^*\cup\{x_e: e\in E^*\}\cup\left(\bigcup_{j\in[i-1]}N_j\right)\right).\]
Note that for $t_i, t_j\in V^*$ where $i\neq j$, it must be the case that $N_{i}$ and $N_{j}$ are disjoint.
Therefore, the subgraph of $G$ induced by the set
\[A=V^*\cup\{x_e: e\in E^*\}\cup \left(\bigcup_{i\in [n]} N_i\right)\]
contains $T$ as a subgraph, which is a contradiction.
Hence, $\cH$ has no $T^*$-skeleton.
\end{proof}

By Claim~\ref{claim:noT'}, $G[V_B]$ is $T'$-free, and by the induction hypothesis, $G[V_B]$ has a $c(T')$-star coloring $\varphi^1_B$.
By Claim~\ref{claim:noT*}, $\cH$ has no $T^*$-skeleton, and so by Lemma~\ref{lem:rainbow}, it follows that $\cH$ has a rainbow $q(|V(T^*)|,d)$-coloring
$\varphi_B^2$.
We define a coloring $\varphi$ of $G$ by  letting $\varphi(x)=\varphi_S(x)$ if $x\in V_S$ and  $\varphi(x)=(\varphi^1_B(x), \varphi_B^2(x))$ if $x\in V_B$.
Note that $\varphi$ uses $c(T)$ colors.

Since $\varphi_S$ and $\varphi^1_B$ are star colorings, each of $G[V_S]$ and $G[V_B]$ has no bicolored $P_4$ under $\varphi$.
Moreover, since the colors used by $\varphi$ on $V_S$ and $V_B$ are distinct, the only possible bicolored $P_4:v_1v_2v_3v_4$ under $\varphi$ is when either $v_1,v_3\in V_S$ and $v_2, v_4\in V_B$ or $v_1,v_3\in V_B$ and $v_2, v_4\in V_S$.
Yet, two vertices of $V_B$ with a common neighbor in $V_S$ receive distinct colors by $\varphi$ since they have different colors in $\varphi^2_B$.
This implies that there is no bicolored $P_4$ under $\varphi$, and hence $\varphi$ is a star $c(T)$-coloring of $G$.
This completes the proof.
\end{proof}

\section{Remarks and future research directions}\label{sec:remarks}

In this paper, we initiated the study of determining which classes of $F$-free graphs have bounded $H$-avoiding chromatic number for a given connected bipartite graph $H$.
 We succeed in finding such a characterization for $F$ when $H$ is in a large class of connected bipartite graphs.
For a general bipartite graph $H$, characterizing such graphs $F$ seems difficult to determine.
Yet, it would be interesting to complete the characterization.

\begin{prob}\label{ques-main}
Let $H$ be a connected bipartite graph.
Characterize all graphs $F$ such that the class of $F$-free graphs has bounded $H$-avoiding chromatic number.
\end{prob}

Note that this paper completely solved Problem~\ref{ques-main} when $H$ is a bipartite graph where vertices in one part has degree at most $2$.
An interesting unknown case of this problem is when $H=K_{3, 3}$, and we explicitly state this below.

\begin{prob}\label{ques-k33}
Characterize all graphs $F$ such that the class of $F$-free graphs has bounded $K_{3,3}$-avoiding chromatic number.
\end{prob}

We also ask the list version of $H$-avoiding coloring;
we omit the precise definition of the list $H$-avoiding chromatic number.
Our proof for the if part of Theorem~\ref{thm-main} does not extended easily to the list version, even when $H=P_4$.
It currently seems out of reach, but it would indeed be an intriguing result.

\begin{prob}
Let $H$ be a connected bipartite graph.
For which graphs $F$, does the class of $F$-free graphs has bounded list $H$-avoiding chromatic number?
\end{prob}

Note that the case when $H=P_4$ is of particular interest, as this would be a strengthening of Proposition~\ref{lem:free-bounded} from the star chromatic number to the list star chromatic number.

\begin{prob}
For every forest $F$ when each pair of $3^+$-vertices has even distance, does the class of $F$-free graphs have bounded list star chromatic number?
\end{prob}

\section*{Acknowledgements}
This work was done during the 2nd Korean Early Career Researcher Workshop in Combinatorics.
Ilkyoo Choi was supported by Basic Science Research Program through the National Research Foundation of Korea (NRF) funded by the Ministry of Education (NRF-2018R1D1A1B07043049), and also by Hankuk University of Foreign Studies Research Fund.
Ringi Kim was  supported by the National Research Foundation of Korea (NRF) grant funded by the Korea government (MSIP) (NRF-2018R1C1B6003786).
Boram Park was supported by Basic Science Research Program through the National Research Foundation of Korea (NRF) funded by the Ministry of  Science, ICT \& Future Planning (NRF-2018R1C1B6003577).

\bibliographystyle{plain}
\bibliography{bounded_H-avoiding_coloring}

 \end{document}